\documentclass[a4paper,11pt, reqno, oneside]{amsart}

\usepackage[utf8]{inputenc}
\usepackage{amsmath, amsfonts, amssymb, amsthm, mathtools, bbm, enumerate}

\usepackage[all]{xy}
\usepackage[height=22.5cm, width=15.5cm]{geometry}
\usepackage[active]{srcltx}

\theoremstyle{plain}
\newtheorem{thm}{Theorem}[section]
\newtheorem{lemma}[thm]{Lemma}
\newtheorem{prop}[thm]{Proposition}
\newtheorem{cor}[thm]{Corollary}

\theoremstyle{definition}
\newtheorem{defi}[thm]{Definition}
\newtheorem{rmk}[thm]{Remark}
\newtheorem{ex}[thm]{Example}

\newcommand{\R}{\mathbb{R}}
\newcommand{\C}{\mathbb{C}}

\newcommand{\K}{\mathbb{K}}

\newcommand{\g}{\mathfrak{g}}
\newcommand{\h}{\mathfrak{h}}

\DeclareMathOperator{\Aut}{Aut}
\DeclareMathOperator{\Der}{Der}
\DeclareMathOperator{\id}{id}
\DeclareMathOperator{\Sym}{Sym}

\DeclareMathOperator{\Tr}{Tr}

\DeclareMathOperator{\GL}{GL}

\DeclareMathOperator{\SO}{SO}
\DeclareMathOperator{\End}{End}

\title{Conjugacy classes of $n$-tuples in semi-simple Jordan algebras}
\author{Hannah Bergner}
\address{Fakult\"at f\"ur Mathematik, Ruhr-Universit\"at Bochum, 
Universit\"atsstr. 150, D-44780 Bochum, Germany}
\email{Hannah.Bergner-C9q@rub.de}

\thanks{Financial support by SFB/TR 12 ``Symmetries and Universality 
  in Mesoscopic Systems'' of the DFG is gratefully acknowledged.}

\begin{document}

\begin{abstract}
Let $J$ be a finite-dimensional semi-simple Jordan algebra over an 
algebraically closed field  of characteristic $0$.
In this article, the diagonal action of the automorphism group
of $J$ on the $n$-fold product $J\times\ldots \times J$ is studied.
In particular, it is shown that the orbit through an $n$-tuple 
$x=(x_1,,\ldots,x_n)$ is closed if and only if the 
Jordan subalgebra generated by the elements $x_1,\ldots, x_n$ is semi-simple.
\end{abstract}

\maketitle

\section{Introduction}
Let $J$ be a finite-dimensional semi-simple Jordan algebra over an algebraically
closed field~$\mathbb K$ of characteristic $\mathrm{char}(\mathbb K)=0$.
Denote by $\Aut(J)$ the automorphism group of $J$, which naturally acts on $J$.
In this article, geometric properties of the diagonal action of $\Aut(J)$ on the
$n$-fold product of $J\times \ldots \times J$ are studied.
In particular, an algebraic characterization of the closed orbits is given.

In the case of a reductive linear algebraic group $G$ with Lie algebra 
$\mathfrak g$, 
R.~W.~Richardson studied geometric properties of the diagonal action of 
$G$ on the $n$-fold product of the Lie algebra $\mathfrak g$, see
\cite{Richardson}.
In particular, he showed that the orbit through an $n$-tuple 
$x=(x_1,\ldots,x_n)$ is closed if and only if the Lie algebra generated by 
$x_1,\ldots, x_n$ is reductive in $\mathfrak g$. Moreover, the closure of
$G\cdot x$ contains $0$ precisely if $x_1,\ldots,x_n$ are contained in a
nilpotent subalgebra of~$\mathfrak g$ which consists of nilpotent elements. 
These results obtained by R.~W.~Richardson generalize the results in 
the case $n=1$,
i.e. the adjoint action of a reductive linear algebraic group on its Lie 
algebra. There, an orbit through an element $x$ is closed if and only if the 
element $x$ is semi-simple. Furthermore, if $x=x_s+x_n$ is the decomposition of 
$x$ into its semi-simple and nilpotent part, then the orbit through $x_s$ is 
the unique closed orbit in the closure of the orbit through $x$.

Similar results can be obtained in the case of Jordan algebras:

\begin{thm}[see Section 6]\label{thm: thm1}
Let $J$ be a semi-simple Jordan algebra over an algebraically closed field of 
characteristic $0$, with automorphism group $\Aut(J)$. 
Consider the diagonal action of $\Aut(J)$ on the $n$-fold product 
$J\times \ldots \times J$, and let 
$x=(x_1,\ldots,x_n)\in J\times \ldots\times J$. Then:
\begin{enumerate}
 \item  The $\Aut(J)$-orbit through the $n$-tuple $x$ is closed
 if and only if the subalgebra $A(x)$ generated by $x_1,\ldots ,x_n$ is
 semi-simple.
 \item The closure of the orbit $\Aut(J)\cdot x$ contains $0$ if and only if 
 $A(x)$ is solvable.
\end{enumerate}
\end{thm}

Moreover, the closed orbit in the closure of an orbit $\Aut(J)\cdot x$
can be described in the following way:

\begin{thm}[see Section 7]\label{thm: thm2}
Let $x=(x_1,\ldots, x_n)\in J\times\ldots \times J$, and let $A(x)$ be the 
subalgebra generated by $x_1,\ldots ,x_n$. Moreover, let $R(x)$ be the radical
of $A(x)$, $S(x)\subseteq A(x)$ a semi-simple subalgebra with $A(x)=S(x)+R(x)$,
and $s_j\in S(x)$, $r_j\in R(x)$ such that $x_j=s_j+r_j$.
Then $\Aut(J)\cdot s=\Aut(J)\cdot (s_1,\ldots, s_n)$ is the
unique closed orbit in the closure of $\Aut(J)\cdot x$.
\end{thm}

The methods to prove the above statements are similar to those used by
R.W. Richardson in \cite{Richardson} in the case of Lie groups and Lie 
algebras.

The key ingredient for the proof is the existence of appropriate
one-parameter subgroups. 
Firstly, it is needed that if a reductive linear algebraic group $G$ acts on an 
affine variety $X$, then for any $x\in X$ there exists a one-parameter
subgroup $\lambda:\mathbb K^\times \rightarrow G$ such that 
$\lim_{t\to 0} \lambda(t)\cdot x$ exists and is contained in the unique closed
orbit in the closure of the orbit $G\cdot x$ (see \cite{Birkes}).

Secondly, the following proposition on the existence of one-parameter subgroups
in the automorphism group of a semi-simple Jordan algebra $J$ satisfying
suitable properties with respect to subalgebras of $J$ is needed:

\begin{prop}[see Section 6]\label{prop: ex of Jordan grading intro}
 Let $J$ be a semi-simple Jordan algebra. Let $A$ be a subalgebra of $J$,
 $R$ the radical of $A$, and let $S$ be a semi-simple subalgebra such that
  $A=S+R$.
  Then there exists a one-parameter subgroup 
 $\lambda:\mathbb K^\times \rightarrow \Aut(J)$ such that
 \begin{align*}
  &A\subseteq J_{\geq 0}(\lambda),\\
  &S\subseteq J_0(\lambda),\text{ and }\\
  & R\subseteq J_{>0}(\lambda),
 \end{align*}
 where 
 $J_k(\lambda)=\{z\in J\,|\,\lambda(t)\cdot z=t^kz
 \text{ for all }t\in \K^\times\}$ for $k\in \mathbb Z$,
 $J_{\geq 0}(\lambda)=\bigoplus_{k\geq 0}J_k(\lambda)$, and
 $J_{>0}(\lambda)=\bigoplus_{k>0}J_k(\lambda)$.
\end{prop}
To prove this proposition, the structure
of the Lie subalgebra 
$$[A,A]\oplus A=\left\{\left.\sum_{i=1}^l [L(a_i),L(b_i)]+L(a)\right|
a_i,b_i,a\in A\right\}\subseteq \mathfrak{gl}(J)$$ 
is studied, where $L(z)$ denotes the left-multiplication $J\rightarrow J$
with $z$ for any $z\in J$.
Then, a result on the existence of one-parameter subgroups in 
Lie algebra invariant under an involution is applied 
(see \cite{Richardson}, Proposition 12.4).

I would like to thank E.~B.~Vinberg for drawing my attention to 
the results of R.~W.~Richardson (\cite{Richardson}) and raising the question
whether similar statements hold true in the case of Jordan algebras. 
Moreover, I am grateful for his interest in this work and the invitation 
to Lomonosov Moscow State University in April 2014, where part of this paper 
was written.

\section{Basic definitions}
In the following, a few basic definitions related to Jordan algebras are 
recalled,
see also e.g. \cite{BraunKoecher} and \cite{FarautKoranyi}.

\begin{defi}[Jordan algebra]
 A Jordan algebra is a vector space $J$ (over a field $\K$) together with a 
 bilinear map $J\times J\rightarrow J$, $(x,y)\mapsto x\circ y$, satisfying 
 \begin{enumerate}[$(i)$]
  \item $x\circ y=y\circ x$ and 
  \item $x\circ (x^2\circ y)=x^2\circ(x\circ y)$
 \end{enumerate}
 for all $x,y\in J$.
 
 For any $x\in V$ let $L(x):J\rightarrow J$ denote the linear map defined by 
$L(x)(y)= x\circ y$.
\end{defi}

The second property in the definition of a Jordan algebra is equivalent to 
requiring that $L(x)$ and $L(x^2)$ commute for any $x\in J$, i.e. 
$[L(x),L(x^2)]=0$.

In the following, we will assume the field $\K$ to have characteristic 
$\mathrm{char}(\K)=0$.
Moreover, the dimension of the vector space $J$ of a Jordan algebra
will be assumed to be finite.

\begin{ex}
 If $A$ is any associative algebra (over $\K$), then a Jordan algebra structure 
 on $A$ can be defined by setting
 \begin{equation*}
  x\circ y=\frac 1 2 (xy+yx).
 \end{equation*}
\end{ex}
A Jordan algebra $J$ is called special if there exists a finite-dimensional 
 associative algebra~$A$ such that $J$ is isomorphic to a subalgebra of the 
 Jordan algebra $(A,\circ)$.

\begin{defi}
 A derivation of a Jordan algebra $J$ is a linear transformation
 $D:J\rightarrow J$
 such that $D(x\circ y)=Dx\circ y+x\circ Dy$ for all $x,y\in J$,
 which is equivalent to $[D, L(x)]=L(Dx)$ for any $x\in J$.
 
 Let $\Der(J)$ denote the set of derivations of $J$. This is a Lie algebra
 with respect to the usual bracket $[D_1,D_2]=D_1D_2-D_2D_1$.
\end{defi}

If $x,y\in J$, then the commutator $[L(x),L(y)]:J\rightarrow J$ is a derivation
of $J$ (see e.g. \cite{FarautKoranyi}, Proposition II.4.1).

\begin{defi}
 A derivation $D$ of a Jordan algebra $J$ is called inner if there are elements
 $x_i,y_i$, $i=1,\ldots, l$, such that $D=\sum_{i=1}^{l} [L(x_i),L(y_i)]$.
\end{defi}

\begin{defi}[quadratic representation]
 Let $J$ be Jordan algebra. For any $x\in J$ define a
 linear map $P(x):J\rightarrow J$, $P(x)=2L(x)^2-L(x^2)$.
 The assignment $P:J\rightarrow \End(J)$, $x\mapsto P(x)$,
 is called the quadratic representation of $J$.
\end{defi}

\section{Structure Theory of Jordan Algebras}
In this section, facts about the structure theory of Jordan algebras, which
will be needed later, are collected. For more details see \cite{Albert}
and \cite{Jacobson}.

Using the results of \cite{Albert} on the structure theory of Jordan algebras,
the radical of a Jordan algebra can be defined as follows:

\begin{defi}
 The radical of a Jordan algebra $J$ is the greatest ideal of $J$ consisting 
 of nilpotent elements. 
 
\end{defi}

\begin{defi}
 A Jordan algebra whose radical vanishes is called semi-simple. 
 It is called solvable if it equals its own radical.
\end{defi}

\begin{defi}
 Let $J$ be a Jordan algebra.
 Define the trace form $\tau=\tau_J$ on $J$ by setting
 $\tau(x,y)=\Tr (L(x\circ y))$ for any $x,y\in J$.
\end{defi}

\begin{defi}
 An inner product $(\cdot, \cdot)$ on a Jordan algebra $J$ is called 
 associative if $(x\circ y,z)=(x,y\circ z)$ for all $x,y,z\in J$.
 This is equivalent to requiring that each left multiplication $L(x)$, $x\in J$,
 is symmetric with respect to the inner product.
\end{defi}

\begin{rmk}[c.f. Proposition II.4.3 in \cite{FarautKoranyi}]
 The trace form $\tau$ on a Jordan algebra $J$ is an associative symmetric 
 bilinear form, since 
 $$\tau(x\circ z,y)-\tau(x,z\circ y)=\Tr(L((x\circ z)\circ y-x\circ(z\circ y)))
 =\Tr([[L(y),L(x)],L(z)])=0 .$$
\end{rmk}

\begin{thm}[see \cite{Albert}, $\S$ 8]
 A Jordan algebra $J$ is semi-simple if and only if its trace form $\tau_J$ 
 is non-degenerate.
\end{thm}

\begin{thm}[see e.g. \cite{Koecher}, Chapter III, Theorem 10]
 A Jordan algebra $J$ containing an identity element $e$ is semi-simple 
 if there exists a non-degenerate associative symmetric bilinear form on $J$.
\end{thm}

\begin{prop}[see \cite{Jacobson}, Theorem 9.2]
 Let $J$ be a Jordan algebra, and $S$ a semi-simple subalgebra of $J$. Then any
 derivation $D:S\rightarrow J$ can be extended to an inner 
 derivation of $J$, i.e. there exist elements $a_i,b_i\in J$, $i=1,\ldots,l$, 
 such that $D=\sum_{i=1}^{l}[L(a_i),L(b_i)]$.
\end{prop}

\begin{defi}
 Let $J$ be a Jordan algebra. Two subalgebras $A_1$ and $A_2$ of $J$ are called
 strictly conjugate if there exists nilpotent derivations 
 $D_1,\ldots,D_k$ of $V$ such that 
 $$(\exp(D_1)\circ\ldots\circ\exp(D_k))(A_1)=A_2.$$
\end{defi}

For Jordan algebras, there is an analogue of the Levi-Malcev theorem for Lie 
algebras:
\begin{thm}[c.f. \cite{Penico}, and \cite{Jacobson}, Theorem 9.3]\label{thm: 
Levi decomp}
 Let $J$ be a Jordan algebra (over an algebraically closed field)
 with radical $R$. Then there exists a semi-simple 
subalgebra $S$ of $J$ such that $J=S\oplus R$ as vector spaces,
 $S\cap R={0}$, and $J/R\cong S$, i.e. $J$ is the semi-direct product of $S$ 
with $R$.
 
 If $S'$ is any semi-simple subalgebra of $J$, then $S'$ is strictly conjugate 
 to a subalgebra of $S$. 
 Moreover, the automorphism mapping $S'$ onto a subalgebra of $S$ can be chosen 
 to be of the form $\exp(D)$ for an inner nilpotent derivation $D$ of $J$ with 
 $$D=\sum_{i=1}^l [L(z_i),L(r_i)],\,z_i\in J,\,r_i\in R.$$
\end{thm}

\begin{rmk}\label{rmk: nilpotent derivation}
 If $J=S+ R$ is a Jordan algebra with radical $R$ and $S\subseteq J$ 
 semi-simple, then any derivation of the form 
  $D=\sum_{i=1}^l [L(z_i),L(r_i)],\,z_i\in J,\,r_i\in R$
  is nilpotent.
  
  For any ideal $A\subseteq J$ define  $A^{(1)}=A$, 
  $A^{(k)}=A^{(k-1)}A^{(k-1)}=\{\sum_ia_ib_i|\,a_i,b_i\in A^{(k-1)}\}$.
  Since $R$ is the radical of $J$, there is $n\in \mathbb N$ with 
  $R^{(n)}=\{0\}$.
  The chain $A=A^{(1)}\supseteq A^{(2)}\supseteq \ldots$ is a chain of 
  subalgebras of $J$, but $A^{(k)}$ is in general not
  an ideal in $J$ due to the non-associativity of $J$.
  Therefore, one cannot expect to have $D(R^{(k)})\subseteq R^{(k+1)}$ in 
  general, and the nilpotence of $D$ does not directly follow from the 
  existence of $n\in \mathbb N$ with $R^{(n)}=\{0\}$.
  
  For a proof of the nilpotence of $D$ see e.g. Corollary~8.4 in
  \cite{Jacobson}.
  
\end{rmk}

\section{Automorphism and structure group}

\begin{defi}
 Define the automorphism group of a Jordan algebra $J$ to be
 $\Aut(J)=\{g\in \mathrm{GL}(J)|\,g(x\circ y)=(gx)\circ(gy) \text{ for 
all } x,y\in J\}$.
\end{defi}

The automorphism group of any Jordan algebra $J$ is an algebraic subgroup of 
$\GL(J)$.

\begin{rmk}
 The Lie algebra of $\Aut(J)$ is the Lie algebra $\Der(J)$ of derivations.
 
 Moreover, in the case of a semi-simple Jordan algebra
 all derivations are inner.
\end{rmk}

\begin{defi}
 Let $J$ be a semi-simple Jordan algebra with trace form 
 $\tau$.
 For any invertible linear transformation $g\in \GL(J)$, let $^t g$ denote
 the transpose of $g$ with respect to $\tau$.
 Define the structure group $\mathrm{Str}(J)$ of $J$ to be 
 $\mathrm{Str}(J)=\{g\in \GL(J)|\,P(gx)=gP(x)\,^t g \text{ for any }x\in J\}$,
 and let $\mathfrak{str}(J)$ denote its Lie algebra.
\end{defi}

\begin{rmk}
 If $J$ is semi-simple, then $\Aut(J)$ and $\mathrm{Str}(J)$ are reductive 
 subgroups of $\GL(J)$, see e.g. \cite{Jacobson_book}, Chapter VIII, Theorem 3.
\end{rmk}

\begin{rmk}
 For the Lie algebra $\mathfrak{str}(J)$ of the structure group $J$ we 
have
$$\mathfrak{str}(J)=\Der(J)\oplus J$$
since every element $X\in \mathfrak{str}(J)$ can be written uniquely as 
$X=D+L(a)$, where $D\in \Der(J)$ and $a\in J$ (see e.g. \cite{FarautKoranyi},
Proposition VIII.2.6). 
The Lie algebra structure on $\Der(J)\oplus J$ is thus given 
by $$[D_1+a_1,D_2+a_2]=([D_1,D_2]+[L(a_1),L(a_2)])+(D_1(a_2)-D_2(a_1))$$
for $D_1,D_2\in \Der(J)$ and $a_1,a_2\in J$.

Note that $\Der(A)\oplus A$ is a Lie algebra with this definition of 
the bracket for any Jordan algebra $A$; c.f. Exercise 2, Chapter I in
\cite{Satake}.
The homomorphism of Lie algebras
$\Der(A)\oplus A\rightarrow \mathfrak{gl}(A)$, $D+a\mapsto D+L(a)$,
is injective for any Jordan algebra $A$ with an identity element.
\end{rmk}

\begin{rmk}
 If $J$ is a complex semi-simple Jordan algebra, then it is the complexification
 of a real Euclidean Jordan algebra $J_\mathbb R$ 
 (c.f. Theorem VIII.5.2 in \cite{FarautKoranyi}),
 where a real Jordan algebra is called Euclidean (or compact) if its trace 
 form $\tau$ is positive-definite. 
 The automorphism group of the Euclidean Jordan algebra $J_\R$ is
 a closed subgroup of the orthogonal group $O(J_\R, \tau)$ and thus compact.
 The connected component of $\Aut(J)$ is the complexification of the
 connected component of $\Aut(J_\mathbb R)$, which shows again that
 $\Aut(J)$ is a complex reductive Lie group.
 \end{rmk}

\begin{prop}[\cite{FarautKoranyi}, Proposition VIII.2.3]
The automorphism group $\Aut(J)$ is a subgroup of the structure group 
$\mathrm{Str}(J)$. Furthermore, if $e$ denotes the identity element of $J$,
then $\Aut(J)$ is the isotropy of $\mathrm{Str}(J)$ in $e$.
\end{prop}

\section{One-parameter subgroups}
In the following, let $\mathbb K$ be algebraically closed,
and let $\mathbb K^\times$ denote the set of non-zero elements in $\mathbb K$.

The following result on the existence of one-parameter subgroups
will be frequently used in the next sections 
when studying closed orbits and orbits closures of the automorphism group 
of a Jordan algebra in the $n$-fold product of the Jordan algebra.

\begin{thm}[see \cite{Birkes}, Theorem 4.2]\label{thm: ex of one-param}
 Let $G$ be a reductive algebraic group and $\rho:G\rightarrow \GL(V)$
 a representation of $G$.
 Let $x\in V$ and $Y$ be the unique closed orbit in the closure of $G\cdot x$.
 Then there exists a one-parameter subgroup 
 $\lambda:\K^\times\rightarrow G$
 such that $\lim_{t\to 0}\lambda(t)\cdot x$ exists and 
 $\lim_{t\to 0}\lambda(t)\cdot x=y\in Y$.
\end{thm}

 Let $J$ be a semi-simple Jordan algebra and let 
 $\lambda:\mathbb K^\times\rightarrow \Aut(J)$ be a one-parameter 
 subgroup.
 As the irreducible representations of $\mathbb K^\times$ are given by
 $t\mapsto t^k$, $k\in \mathbb Z$, we have 
 the decomposition 
 $$ J=\bigoplus_{k\in \mathbb Z} J_k(\lambda),$$
 where $J_k(\lambda)=\{z\in J\,|\,\lambda(t)\cdot z=t^kz 
 \text{ for all } t\in \mathbb K^\times\}$. 
 This defines a grading of the Jordan algebra $J$ in the sense that
 $J_k(\lambda)\circ J_l(\lambda)\subseteq J_{k+l}(\lambda)$.
 We set 
 \begin{align*}
  &J_{\geq 0}(\lambda)=\bigoplus_{k\geq 0} J_k(\lambda)
  =\{z\in J\, |\,\lim_{t\to 0}\lambda(t)\cdot z \text{ exits}\},\\
  &J_0(\lambda)=\{z\in J\,|\, \lambda(t)\cdot z=z\text{ for all } t\in 
  \mathbb K^\times \},\text{ and }\\
  &J_{>0}(\lambda)=\bigoplus_{k>0}J_k(\lambda)=\{z\in J_{\geq 0}(\lambda)\,|\,
  \lim_{t\to 0}\lambda(t)\cdot z =0\}.
\end{align*}

\begin{rmk}
 Note that the subspaces $J_{\geq 0}(\lambda)$, $J_{>0}(\lambda)$, and 
 $J_0(\lambda)$ are subalgebras of $J$. 
 
 The subalgebra $J_0(\lambda)$ is semi-simple since the trace form $\tau_J$ of
 $J$ restricted to $J_0(\lambda)$ is a non-degenerate associative symmetric
 bilinear form and the identity $e$ is contained in $J_0(\lambda)$.
 
 Furthermore, $J_{>0}(\lambda)$ is the radical of $J_{\geq 0}(\lambda)$
 and $J_{\geq 0}(\lambda)$ is the semi-direct product of 
 $J_0(\lambda)$ and $J_{>0}(\lambda)$.
\end{rmk}

\begin{defi}\label{defi: limit subspaces}
 For any one-parameter subgroup $\lambda:\mathbb K^\times\rightarrow \Aut(J)$, 
 define moreover
 the following subspaces of 
 $\mathfrak h=\mathfrak{str}(J)=\Der(J)\oplus J$:
 
 \begin{align*}
  &\mathfrak p(\lambda)=\mathfrak h_{\geq 0}(\lambda)=\{X\in\mathfrak h\,|\,
     \lim_{t\to 0}\lambda(t)\cdot X\text{ exits}\}\\
  &\mathfrak h^\lambda=\mathfrak h_0(\lambda)=\{X\in \mathfrak h\,|\, 
     \lambda(t)\cdot X=X\text{ for all } t\in \mathbb K^\times\}\\
  &\mathfrak u(\lambda)=\mathfrak h_{>0}(\lambda)=\{X\in \mathfrak p(\lambda)\,
  |\,\lim_{t\to 0}\lambda(t)\cdot X=0\}
 \end{align*}
\end{defi}

\section{Closed orbits of the automorphism group}
In the following, let $J$ denote a semi-simple Jordan algebra over $\mathbb K$,
where $\mathbb K$ is algebraically closed.

The automorphism group $\Aut(J)$ acts diagonally on the $n$-fold product 
$J\times\ldots\times J$ of $J$:
$$\Aut(J)\times (J\times\ldots\times J)\rightarrow J,
\,g(x_1,\ldots,x_n)=(gx_1,\ldots,gx_n).$$
The goal is now to characterize the closed orbits in $J\times\ldots\times J$ 
under this action. 
An $\Aut(J)$-orbit through an $n$-tuple $(x_1,\ldots,x_n)$ is closed if and 
only if the orbit of the connected group $G=\Aut(J)^\circ$ through 
$(x_1,\ldots,x_n)$ is closed since $\Aut(J)$ is an algebraic 
subgroup of $\mathrm{GL}(J)$ and hence has finitely many connected 
components.

Let $H$ denote the connected component of the structure
group $\mathrm{Str}(J)$ of $J$, and $\mathfrak h=\mathfrak{str}(J)$
its Lie algebra.

\begin{defi}
 For any $n$-tuple $x=(x_1,\ldots,x_n)$ in $J\times\ldots\times J$, 
 define $A(x)$ to be the subalgebra of $J$ generated by
 $x_1,\ldots,x_n$.
 Moreover, let $\mathfrak l(x)$ be the Lie subalgebra
  \begin{equation*}
  [A(x),A(x)]\oplus A(x)
  =\left\{\left.\sum_{i=1}^k[L(y_i),L(y_i')]+z\right|
  \,y_i,y_i',z\in A(x)\right\}
  \subseteq \Der(J)\oplus J=\mathfrak h.
 \end{equation*}
\end{defi}

\begin{ex}
 Let $J=\Sym_2(\C)$.
 We have $G=\Aut(J)^\circ=\SO_2(\C)$ and $H=\GL_2(\C)$ where $G$ and $H$ act via
 $g\cdot X=gX\,^tg$ for $X\in \Sym_2(\C)$.
\begin{enumerate}[(i)]
 \item  Let $n=2$, and $x=\left(\left(\begin{smallmatrix}
           1&i\\ i&-1 
          \end{smallmatrix}\right),
          \left(\begin{smallmatrix}
           1&-i\\-i&-1
          \end{smallmatrix}\right)\right)$.
          Then $A(x)=\Sym_2(\mathbb C)$, and 
          $\mathfrak l (x)=\mathfrak{gl}_2(\mathbb C)$.
          Remark that $A(x)$ is a simple Jordan algebra even though 
          both $\left(\begin{smallmatrix}
           1&i\\ i&-1 
          \end{smallmatrix}\right)$ and 
          $\left(\begin{smallmatrix}
           1&-i\\-i&-1
          \end{smallmatrix}\right)$
          are nilpotent elements of $\Sym_2(\mathbb C)$.
          
 \item  Let $n=2$, and
 $x=\left(\left(\begin{smallmatrix}
           1&0\\ 0&1 
          \end{smallmatrix}\right),
          \left(\begin{smallmatrix}
           1&i\\i&-1
          \end{smallmatrix}\right)\right)$. 
          Then $A(x)=\C\left(\begin{smallmatrix}
                     1&0\\ 0&1 
                    \end{smallmatrix}\right)
                \oplus \C
                \left(\begin{smallmatrix}
           1&i\\i&-1
          \end{smallmatrix}\right)$
 and since $\left[L\left(\begin{smallmatrix}
           1&0\\ 0&1 
          \end{smallmatrix}\right),
        L\left(\begin{smallmatrix}
           1&i\\i&-1
          \end{smallmatrix}\right)\right]=0$
 we also get $\mathfrak l(x)= \C\left(\begin{smallmatrix}
                     1&0\\ 0&1 
                    \end{smallmatrix}\right)
                \oplus \C
                \left(\begin{smallmatrix}
           1&i\\i&-1
          \end{smallmatrix}\right)$.
\end{enumerate}
\end{ex}

\begin{defi}
 Let $G$ be a linear algebraic group with Lie algebra $\g$.
 An algebraic Lie subalgebra $\mathfrak c$ of $\g$ is a Lie subalgebra such
 that there exists an algebraic subgroup $C$ of $G$ with Lie algebra 
 $\mathfrak c$.
\end{defi}

\begin{defi}
 Let $G$ be a reductive linear algebraic group, and $\g$ its Lie algebra.
 A Lie subalgebra $\mathfrak c$ of $\mathfrak g$ is called reductive in 
 $\mathfrak g$ if the adjoint representation of $\g$ restricted to $\mathfrak c$
 is completely reducible.
\end{defi}

\begin{rmk}
 A Lie subalgebra $\mathfrak c\subset \mathfrak g$ is reductive in
 $\mathfrak g$ if and only if its algebraic hull, the smallest algebraic
 Lie subalgebra of $\mathfrak g$ containing $\mathfrak c$, is reductive in 
 $\mathfrak g$; c.f. \cite{Hochschild_AlgGroups}, $\S$ VIII.3.
\end{rmk}

Define an involution $\Theta$ on $H$ by setting 
$$\Theta: H\rightarrow H,\,\Theta(h)=\,^th^{-1},$$ 
where the transpose $^th$ of the element $h$ is taken with respect to the 
trace form $\tau=\tau_J$ of $J$.
This map $\Theta$ is a group automorphism of $H$.
The differential of $\Theta:H\rightarrow H$ is given by 
$$\theta:\h\rightarrow \h,\,\theta(X)=-^tX.$$
Its fixed point set is exactly $\Der(J)$, which is the Lie algebra 
$\g$ of the automorphism group $\Aut(J)$, and moreover we have 
$\theta|_{J}=-\id_{J}$.
Therefore, the connected component of the fixed point group 
$H^\Theta$ is precisely the connected component $G$ of the automorphism group 
$\Aut(J)$.

 Let $A$ be any subalgebra of $J$ and $A=S+R$ be a Levi decomposition of $A$, 
 where $S$ is a semi-simple 
 subalgebra of $A$ and $R$ is the radical 
 (c.f. Theorem \ref{thm: Levi decomp}).
 Next, the structure of  
 $[A,A]\oplus A=\left\{ \left.\sum_{i=1}^k[L(y_i),L(y_i')]
 +L(z)\right| \,y_i,y_i',z\in A\right\}$ as a subset of 
 $\h=\mathfrak{str}(J)\subset\mathfrak{gl}(J)$ is examined.
 Note that $[A,A]\oplus A$ is $\theta$-invariant.
 
 The centre $Z_S$ of $S$ is defined to be the subset
 of $S$ 
 consisting of the elements which operator-commute
 with all other elements of $S$ in $S$, 
 i.e. $s\in Z_S$ if and only if the the derivation 
 $[L(s),L(t)]$ vanishes (on $S$) for every $t\in S$.
 
\begin{rmk}
 Let $S=\bigoplus_i S_i$ a decomposition of $S$ into simple subalgebras $S_i$
 (c.f. Theorem~11 in \cite{Albert}).
 By Satz 5.1, Kapitel I, of \cite{BraunKoecher} the centre $Z_{S_i}$ of each 
 simple subalgebra $S_i$ is $\mathbb K e_i$, where $e_i$ is the identity element 
 of $S_i$.
 Therefore, the centre of $S$ is 
 $Z_S=\bigoplus_i Z_{S_i}=\bigoplus_i \mathbb K e_i$.
 
 The structure Lie algebra 
 $\mathfrak{str}(S)\cong\Der(S)\oplus S$ of $S$ is reductive in 
 $\mathfrak{gl}(S)$, all derivations $X\in \Der(S)$ are inner, and its centre 
 is $Z_S$.
 
 Let $S'$ be the subspace of $S$ defined by 
 $S'=\Der(S)(S)=\{\sum_i [L(s_i),L(t_i)](s)|\,s_i,t_i,s\in S\}$. 
 Since $\Der(S)\oplus S$ is a reductive Lie algebra, its derived algebra 
 \begin{align*}
  \left[\Der(S)\oplus S,\Der(S)\oplus S\right]
  = \Der(S)\oplus \left(\Der(S)(S)\right)=\Der(S)\oplus S'
 \end{align*}
 is semi-simple and 
$  \Der(S)\oplus S=\left(\Der(S)\oplus S'\right)\oplus Z_S$,
 which implies $S= S'\oplus Z_S$,
 and $\Der(S)=[S'\oplus Z_S,S'\oplus Z_S]=[S',S']
 =\{\sum_i[L(r_i),L(s_i)]|\,r_i,s_i\in S'\}$.
 Furthermore, we have $S'=(Z_S)^\perp=\{z\in J\,|\,\tau_J(z,w)=0\text{ for all } 
 w\in Z_S\}$.
\end{rmk}

\begin{ex}
 Let $S=\Sym_2 (\mathbb C)$. 
 Then $[\Sym_2(\mathbb C),\Sym_2(\mathbb C)]\oplus \Sym_2(\mathbb C)
 \cong\mathfrak{gl}_2(\mathbb C)$ and 
 $[S,S]\cong \mathfrak{so}_2(\mathbb C)$.
 The centre of $S=\Sym_2(\mathbb C)$ is $Z_S=\mathbb C 
 \left(\begin{smallmatrix}
  1&0\\0&1
 \end{smallmatrix}\right)$,
 \begin{align*}
  &S'=\{x\in \Sym_2 (\mathbb C)|\, \Tr(x)=0\}
  =\Sym_2(\mathbb C)\cap \mathfrak{sl}_2(\mathbb C),
 \end{align*}
 and $[S',S']\oplus S'\cong \mathfrak{sl}_2(\mathbb C).$
 In particular, $S'\subset S$ is not a subalgebra of the Jordan algebra $S$. 
\end{ex}

\begin{lemma}
 The Lie subalgebra $[S,S]\oplus S\subset\mathfrak{gl}(J)$ is reductive in 
 $\mathfrak{gl}(J)$. 
\end{lemma}

\begin{proof}
 The structure Lie algebra
 $\mathfrak{str}(S)\cong \Der(S)\oplus S$ of $S\subseteq J$ is reductive.
 The inclusion map $S\hookrightarrow J$ extends to an injective homomorphism of
 Lie algebras 
 $\varphi:\mathfrak{str}(S)\hookrightarrow
 \mathfrak{str}(J)\subset \mathfrak{gl}(J)$, 
 where an element $X+s\in \Der(S)\oplus S$,
 $X=\sum_i[L(r_i),L(s_i)]$, $r_i,s_i\in S$, is mapped to 
 $\sum_i[L(r_i),L(s_i)]+L(s)$, now considered as a transformation of the 
 Jordan algebra $J$ (c.f. \cite{Jacobson}, Lemma 8.3).
 The image of $\varphi$ is $[S,S]\oplus S$.
 
 Using the preceding remark, it follows that $[S',S']\oplus S'$ is a
 semi-simple Lie algebra. It thus remains to show 
 that the centre $Z$ of $[S,S]\oplus S$ consists of semi-simple elements.
 If $S=\bigoplus_i S_i$ is again a decomposition of $S$ into simple subalgebras
 $S_i$, and $e_i$ denotes the identity element of $S_i$, 
 the centre $Z$ is spanned by the left multiplications 
 $L(e_i)\in\mathfrak{gl}(J)$.
 Since the element $e_i$ is the identity of $S_i$, it is in particular 
 idempotent, i.e. $e_i^2=e_i$, as an element of $J$. This implies that
 $L(e_i):J\rightarrow J$ is diagonalisable and 
 the only possible eigenvalues of $L(e_i)$ are $1,\frac 1 2 , 0$
 (see e.g. \cite{Albert}, $\S$ 5, Lemma 5). In particular, 
 the transformation $L(e_i):J\rightarrow J$ is semi-simple.
\end{proof}

\begin{rmk}
 The proof of the lemma shows moreover that $[S,S]\oplus S$ is an algebraic
 Lie subalgebra of $\mathfrak{gl}(J)$.
\end{rmk}

\begin{rmk}
 Using the notion of a representation of a Jordan algebra (see e.g.
 \cite{Jacobson}), the lemma could also be proven in the following way:\\
 Since $S$ is a semi-simple Jordan algebra, the representation 
 $S\rightarrow \End(J)$, $x\mapsto L(x)$,
 is completely reducible (c.f. \cite{Jacobson}, $\S$ 8).
 Consequently, the representation of $[S,S]\oplus S\subset 
 \mathfrak{gl}(J)$ on $J$
 is also completely reducible and thus $[S,S]\oplus S$ is 
 reductive in $\mathfrak{gl}(J)$. 
\end{rmk}

\begin{lemma}
 The subspace 
 $[A,R]\oplus R$ is a nilpotent ideal in 
 $[A,A]\oplus A\subseteq \mathfrak{str}(J)$
 and is nilpotent on $J$, i.e. there is $m\in\mathbb N$ such that
 $\{f_1\circ\ldots\circ f_m|\,f_i\in [A,R]\oplus R\}=0$
\end{lemma}

\begin{proof}
 We have $[A,R]\subset [A,R]\oplus R$ and 
 for $a,b\in A$, $r\in R$ we have
 $[a,[L(b),L(r)]]=-(b(ra)-r(ba))\in R$
 since $R$ is an ideal in $A$. Therefore, $[A, [A,R]\oplus R]
 \subset [A,R]\oplus R$
 and consequently  $[[A,A],[A,R]\oplus R]\subseteq [A,R]\oplus R$. Hence, 
 $[A,R]\oplus R$ is an ideal in $[A,A]\oplus A$.
 
 To show that $[A,R]\oplus R\subseteq\h=\mathfrak{str}(J)$ is a nilpotent Lie 
 subalgebra consisting of nilpotent endomorphisms, similar arguments can 
 be used as are needed to 
 show that a derivation of $A$ of the form $\sum_i[L(a_i),L(r_i)]$, $a_i\in A$,
 $r_i\in R$, is nilpotent (c.f. Remark~\ref{rmk: nilpotent derivation}).
 A proof is given for example in  \cite{Jacobson}, Corollary~8.2.
\end{proof}
As a corollary of the preceding lemmata on the structure of $[A,A]\oplus A$
we get the following:

\begin{cor}\label{cor: decomposition in str}
 The Lie subalgebra 
 $[A,A]\oplus A\subseteq \mathfrak{str}(J)\subseteq \mathfrak{gl}(J)$
 might be written as 
 $$[A,A]\oplus A =\left([S,S]\oplus S\right)\oplus \left([A,R]\oplus R\right),$$
 for $A=S+R$, where again $R$ is the radical of $A$ and $S$ is a Levi factor.
 The Lie subalgebra $[S,S]\oplus S$ is reductive in $\mathfrak{gl}(J)$ 
 and $[A,R]\oplus R$ is an ideal which is nilpotent on $J$.
\end{cor}

\begin{prop} The Lie subalgebra 
 $[A,A]\oplus A\subseteq \mathfrak{str}(J)\subseteq 
 \mathfrak {gl}(J)$ is algebraic.
 \end{prop}
 
\begin{proof}
 As remarked before, the Lie subalgebra $[S,S]\oplus S\subset \mathfrak{gl}(J)$
 is algebraic.
 Moreover, the Lie subalgebra $[A,R]\oplus R\subseteq \mathfrak{gl}(J)$
 is an algebraic Lie subalgebra since it is nilpotent on $J$.
 
 As $[A,A]\oplus A
 =([A,R]\oplus R)\oplus ([S,S]\oplus S)$,
 and both $[A,R]\oplus R$ and $[S,S]\oplus S$ are algebraic, 
 it follows that $[A,A]\oplus A$ is an algebraic Lie subalgebra
 of $\mathfrak{str}(J)\subseteq \mathfrak{gl}(J)$ (c.f. 
 \cite{Hochschild_AlgGroups}, Chapter VIII, Theorem 3.4).
\end{proof}

\begin{prop}\label{prop: ex of Jordan grading}
 For any subalgebra $A\subseteq J$, $A=S+R$, where $R$ is the radical of $A$ and
 $S\subseteq A$ is a semi-simple subalgebra, 
 there exists a one-parameter subgroup 
 $\lambda:\mathbb K^\times \rightarrow \Aut(J)$ such that
 \begin{align*}
  &A\subseteq J_{\geq 0}(\lambda),\\
  &S\subseteq J_0(\lambda),\text{ and }\\
  & R\subseteq J_{>0}(\lambda).
 \end{align*}
\end{prop}

\begin{proof}
 Consider the Lie algebra $\mathfrak l=[A,A]\oplus A\subseteq \mathfrak{gl}(J)$.
  By Corollary~\ref{cor: decomposition in str}, it might be 
  written as 
 $\mathfrak l=([S,S]\oplus S)\oplus ([A,R]\oplus R)$,
 $[S,S]\oplus S$ is reductive in $\mathfrak{gl}(J)$ and
 $[A,R]\oplus R$ is a nilpotent ideal whose action on $J$ is 
 nilpotent. %In particular, $\mathfrak a(x)$ is not linearly reductive.
 
 Let $H$ denote again the connected component of the structure group 
 $\mathrm{Str}(J)$, and $\mathfrak h=\mathfrak{str}(J)$ its Lie algebra.
 The Lie algebra $\mathfrak l\subseteq\mathfrak h=\mathfrak {str}(J)$ is 
 $\theta$-invariant and by Proposition 12.4 in \cite{Richardson} there exists a 
 one-parameter subgroup $\lambda:\mathbb K^\times\rightarrow G=H^{\Theta}$ such 
 that $\mathfrak l\subseteq \mathfrak p(\lambda)=\mathfrak h_{\geq 0}(\lambda)$,
 $[S,S]\oplus S\subseteq \mathfrak h^\lambda=\mathfrak h_0(\lambda)$,
 and 
 $[A, R]\oplus R \subseteq \mathfrak u(\lambda)
 =\mathfrak h_{>0}(\lambda)$,
 and the statement of the proposition follows.
\end{proof}

\begin{thm}
 Let $J$ be a complex semi-simple Jordan algebra and consider the diagonal
 action of the automorphism group $\Aut(J)$ on the $n$-fold product
 $J\times\ldots\times J$. 
 Let $x=(x_1,\ldots,x_n)\in J\times\ldots\times J$.
 Then the orbit $\Aut(J)\cdot x$ is closed if and only if $A(x)$ is a 
 semi-simple Jordan algebra.
\end{thm}

\begin{proof}
 First, let the orbit $\Aut(J)\cdot x$, or equivalently the orbit
 $G\cdot x$ of the connected component $G=\Aut(J)^\circ$, be closed.
 Assume now that $A(x)$ is not semi-simple, and let $A(x)=S(x)+ R(x)$ 
 be a decomposition of $A(x)$ such that $S(x)$ is semi-simple and $R(x)\neq 0$
 is the radical of $A(x)$.
 By Proposition~\ref{prop: ex of Jordan grading}, there is
 a one-parameter subgroup $\mathbb K^\times\rightarrow G$ such that 
 $A\subseteq J_{\geq 0}(\lambda)$, $S\subseteq J_0(\lambda)$, and 
 $R\subseteq J_{>0}(\lambda)$.
 Since $A(x)\subseteq J_{\geq 0}(\lambda)$, we have 
 $x_1,\ldots , x_n\in J_{\geq 0}(\lambda)$ and 
 the limit $\lim_{t\to 0}\lambda(t)\cdot x$ exists.
 As we assumed that the orbit through $x$ is closed, we get
 $y=(y_1,\ldots,y_n)=
 \lim_{t\to 0}\lambda(t)\cdot (x_1,\ldots,x_n) \in G\cdot x$.
 Consider the map 
 $$\psi_\lambda:J_{\geq 0}(\lambda)
 \rightarrow J_0(\lambda),\,
 \psi_\lambda(z)=\lim_{t\to 0} \lambda(t)\cdot z.$$
  Remark that $\psi_\lambda|_{J_0(\lambda)}=\mathrm{id}_{J_0(\lambda)}$ and 
 $\psi_\lambda(r)=0$ for all $r\in {J_{>0}(\lambda)}$.
 By definition we have $\psi_\lambda(x_j)=y_j$, $j=1,\ldots n$, and
 $\psi_\lambda(A(x))=S(x)$ is semi-simple.
 Moreover, $y_j\in S(x)$ for all $j$
 and thus $A(y)\subseteq S(x)$. 
 Since $A(x)=S(x)+R(x)$ is not semi-simple, $R(x)\neq 0$, and the dimension of 
 $A(y)\subseteq S(x)$ is strictly smaller than the dimension of $A(x)$.
 At the same time, $A(x)$ and $A(y)$ are conjugated by an element in $G$
 because $x$ and $y$ lie in the same $G$-orbit, which yields a contradiction.

 Now, let $A(x)$ be semi-simple and suppose now that the $G$-orbit through 
 $x=(x_1,\ldots,x_n)$ is not closed. 
 By Theorem~\ref{thm: ex of one-param} there is a one-parameter group 
 $\lambda:\mathbb K^\times\rightarrow G$ such that 
 $$y=(y_1,\ldots,y_n)=\lim_{t\to 0}\lambda(t)\cdot x
 =\lim_{t\to 0}\lambda(t)\cdot (x_1,\ldots,x_n)$$
 exists and $G\cdot y$ is the unique closed $G$-orbit in the closure of 
 $G\cdot x$.
 Since $y_j=\lim_{t\to 0}\lambda\cdot x_j$, $j=1,\ldots,n$, 
 the subalgebra $A(x)$ is contained in $J_{\geq 0}(\lambda)$.
 The subalgebra $A(x)$ is semi-simple by assumption and by 
 Theorem~\ref{thm: Levi decomp} $A(x)$ is strictly conjugate to a subalgebra of 
 $J_0(\lambda)$ and there exists an inner nilpotent derivation 
 $D=\sum_{i=1}^m[L(z_i),L(u_i)]$,
 $z_i\in J_{\geq 0}(\lambda),\,u_i\in J_{>0}(\lambda)$, of 
 $J_{\geq 0}(\lambda)$ such that $\exp(D)(A(x))\subseteq J_0(\lambda)$.
 Note that $D(J_k(\lambda))\subseteq \bigoplus_{l>k}J_l(\lambda)$
 for any $k$, which yields $\lim_{t\to 0}\lambda(t)\cdot(D(z))=0$ for arbitrary 
 $z\in J_{\geq 0}(\lambda)$.
 
 Consider again the map 
 $\psi_\lambda:J_{\geq 0}(\lambda)\rightarrow J_0(\lambda)$, 
 $z\mapsto \lim_{t\to 0}\lambda(t)\cdot z$. 
 The sum $\exp(D)$ is a finite sum since $D$ is nilpotent and we have 
 \begin{align*}
   \psi_\lambda(\exp(D)(z))=\lim_{t\to 0}\lambda(t)\cdot
   \left( \sum_{k=0}^N \frac{1}{k!} D^k(z)\right) 
   =\lim_{t\to 0}\lambda(t)\cdot z
 =\psi_\lambda(z).
\end{align*}
Using that $\exp(D)(z)\in J_0(\lambda)$ for any $z\in A(x)$,
and thus $\psi_\lambda(\exp(D)(z))=\exp(D)(z)$ for all $z\in A(x)$, we obtain
$\psi_\lambda(z)=\exp(D)(z)$ for arbitrary $z\in A(x)$.
The derivation $D$ of $J_{\geq 0}(\lambda)$ can be extended to a derivation $D$ 
of $J$ because it is an inner derivation. 
Hence $\exp(D)\in G$. We have $x_i\in A(x)$ by definition and now get
$$y_i=\lim_{t\to 0}\lambda(t)\cdot x_i=\psi_\lambda(x_i)=\exp(D)x_i.$$
Therefore, 
$$y=(y_1,\ldots,y_n)=\lim_{t\to 0}\lambda(t)\cdot(x_1,\ldots,x_n)
=\exp(D)(x_1,\ldots,x_n)\in G\cdot x,$$
which is a contradiction. 
\end{proof}

\begin{defi}
 An $n$-tuple $x=(x_1,\ldots,x_n)\in J\times\ldots\times J$
 is called unstable if 
 $0=(0,\ldots,0)\in J\times\ldots\times J$
 is contained in the closure $\overline{\Aut(J)\cdot x}$ of the orbit $\Aut(J)\cdot x$.
\end{defi}

\begin{prop}
 An $n$-tuple $x=(x_1,\ldots,x_n)\in J\times\ldots\times J$
 is unstable if and only if $A(x)$ is a solvable subalgebra of $J$.
\end{prop}

\begin{proof}
 Let $x$ be unstable. Then $0\in \overline{\Aut(J)\cdot x}$ and $0$ is the unique 
 closed orbit in the closure of $\Aut(J)\cdot x$. 
 By Theorem~\ref{thm: ex of one-param} there exists a one-parameter subgroup
 $\lambda:\mathbb K^\times\rightarrow \Aut(J)$ with 
 $\lim_{t\to 0}\lambda(t)\cdot x=0$.
 Therefore, $x_1,\ldots,x_n$ and thus also $A(x)$ are contained in 
 $J_{>0}(\lambda)$. Since $J_{>0}(\lambda)$ is the radical 
 of $J_{\geq 0}(\lambda)\subset J$, it is a solvable subalgebra 
 of $J$. Hence, $A(x)$ is also solvable.
 
 Now, let $A(x)$ be solvable. Then $A(x)$ is equal to its radical $R(x)$, and
 by Proposition~\ref{prop: ex of Jordan grading} there is
 a one-parameter subgroup $\lambda:\mathbb K^\times\rightarrow \Aut(J)$
 such that 
 $A(x)\subseteq J_{>0}(\lambda)$.
 Consequently, $\lim_{t\to 0}\lambda(t)\cdot y=0$ for any $y\in A(x)$ and 
 in particular 
 $\lim_{t\to 0}\lambda(t)\cdot (x_1,\ldots,x_n)=(0,\ldots,0)=0\in
 \overline{\Aut(J)\cdot x}$.
\end{proof}

\begin{rmk}
 Consider the case $n=1$. Every element $x\in J$ can be written as
 $x=x_s+x_n$, $x_s,x_n\in \mathbb C [x]$, where $x_s$ is semi-simple and
 $x_n$ is nilpotent, and this decomposition is unique (see e.g. 
 \cite{BraunKoecher}, Kapitel 1, $\S$ 3).
 It can be shown that the Jordan subalgebra $A(x)$ generated by $x$ is 
 semi-simple (solvable) if and only if $x$ is semi-simple (nilpotent).
 Thus, the orbit $\Aut(J)\cdot x$ is closed (unstable) if and only if 
 $x$ is semi-simple (nilpotent).

\end{rmk}

\begin{rmk}
 Let again $n=1$ and consider the isomorphism
 $\Der(J)\oplus J\rightarrow \mathfrak{str}(J)$, $D+z\mapsto D+L(z)$.
 It already follows from \cite{KostantRallis} (Proposition 3 and Theorem 4)
 that the $\Aut(J)$-orbit through an element
 $L(z)\in L(J)\subset \mathfrak{str}(J)$ is closed (unstable)
 in $L(J)\cong J$ if and only if 
 $L(z)\in \mathfrak{str}(J)\subseteq \mathfrak{gl}(J)$ is a semi-simple
 (nilpotent) endomorphism.
 Using Peirce decomposition in $J$ it follows that the element $z\in J$ 
 is semi-simple as an element of the Jordan algebra if and only if 
 $L(z)\in \mathfrak{str}(J)$ is semi-simple, and by
 Theorem~5, Chapter III, \cite{Koecher}, $z\in J$ is nilpotent precisely
 if the endomorphism $L(z)$ is nilpotent.
\end{rmk}

\begin{rmk}
 The statement in the case $n=1$ could also be shown as follows:\\
 Consider again the isomorphism of Lie algebras 
 $\Der(J)\oplus J\rightarrow \mathfrak{str}(J)\subseteq
 \mathfrak{gl}(J),\,D+z\mapsto D+L(z).$
 As mentioned before, the group $\mathbb Z_2$ acts by automorphisms on the 
 structure group $\mathrm{Str}(J)$ 
 via the involution 
 $\Theta:\mathrm{Str}(J)\rightarrow \mathrm{Str}(J)$, $\Theta(h)=\,^th^{-1},$
 where $\,^th$ denotes again the transpose of $h$ with respect to the trace form
 $\tau_J$ of $J$. Let $\theta:\mathfrak{str}(J)\rightarrow \mathfrak{str}(J)$ 
 again be the differential of $\Theta$. The Lie subalgebra fixed by $\theta$ 
 is $\Der(J)$.
 Note that the Lie algebra spanned by $L(z)$ is stable under the involution 
 $\theta$. It is reductive in $\mathfrak{str}(J)$ 
 if and only if $L(z)$ (or equivalently $z$) is semi-simple, 
 and it is a Lie algebra acting nilpotently on $J$ if and only if 
 $L(z)$ (or equivalently $z$) is nilpotent.
 Since the orbit of the automorphism group $\Aut(J)$ through $L(z)$ is contained
 in the subspace $L(J)\cong J$, the orbit 
 $\Aut(J)\cdot L(z)\subseteq \mathrm{str}(J)$ is closed (unstable) if and only
 if the orbit $\Aut(J)\cdot z\subseteq J$ is closed (unstable).
 Thus, by Theorem 13.2/13.3 in \cite{Richardson}, the orbit $\Aut(J)\cdot z$ is 
 closed (unstable) precisely if $z$ is semi-simple (nilpotent).
\end{rmk}

\begin{rmk}
 Let again $n=1$, let $J$ be a complex semi-simple Jordan algebra, and
 let $J_\R$ be a Euclidean real form of $J$, $J=J_\R\oplus iJ_\R$.
 Denote the complex conjugation of an element $x\in J$ with respect to
 this decomposition by $\bar{x}$.
 If $\tau$ denotes the trace form on $J$, which is the
 $\mathbb C$-linear extension of the trace 
 form of $J_\R$, then $\langle x,y\rangle=\tau(x,\bar{y})$ defines a 
 positive-definite Hermitian product on $J$.
 This Hermitian product is invariant under the automorphism group $\Aut(J_\R)$ 
 of $J_\R$. Let $\mathfrak k$ be the Lie algebra of $\Aut(J_\R)$ and
 denote again by $K$ the connected component of $\Aut(J_\R)$ and
 by $G=K^{\mathbb C}$ the connected component of $\Aut(J)$.
 Then $\mu: J\rightarrow \mathfrak k^*$, 
 $\mu^\xi(x)=\mu(x)(\xi)=\langle \xi(x),x\rangle$
 for $x\in J$, $\xi\in\mathfrak k$, defines a $K$-equivariant moment map.
 
 Assume now that $J_\R$ is simple. 
 Using Lemma~VI.1.1 of \cite{FarautKoranyi}, there is a positive
 constant~$\kappa$ such that 
 $$\tau(D x,y)= \kappa \Tr(L(D x)\circ L(y))$$
 for any derivation 
 $D$ and $x,y \in J$.
 Therefore,
 $$\mu^\xi(x)=\tau (\xi(x), \bar x)=\kappa \Tr(L(\xi(x)\circ L(\bar x))
 =\kappa\Tr([\xi,L(x)]\circ L(\bar x))
 =\kappa  \Tr(\xi\circ [L(x),L(\bar x)]).$$
 The bilinear form on $\Der(J)$ defined by 
 $(D_1,D_2)\mapsto \Tr(D_1\circ D_2)$
 is non-degenerate (c.f. \cite{BraunKoecher}, Kapitel IX, Lemma 3.2),
 and identifying $\mathfrak k^*$ with $\mathfrak k$ via a multiple of this 
 bilinear form
 we get 
 $$\mu(x)=[L(x),L(\bar x)].$$
 The zero level set of $\mu$ is 
 $$\mu^{-1}(0)=\{x\in J\,|\, [L(x),L(\bar x)]=0\}.$$
 Writing $x$ as $x=x_1+ix_2$ with $x_j\in J_{\mathbb R}$,
 $[L(x),L(\bar x)]=0$ is equivalent to $[L(x_1),L(x_2)]=0$.
 If two elements $x_1,x_2\in J_\R$ operator-commute, i.e. 
 $[L(x_1),L(x_2)]=0$, then $x_1$ and $x_2$ are simultaneously diagonalisable,
 meaning that
 there exists a Jordan frame
 $c_1,\ldots, c_k$ of $J_\R$, and $\alpha_j,\beta_j\in \mathbb R$
 such that
 $$x_1=\sum_{j=1}^k \alpha_j c_j\ \text{ and }\ x_2=\sum_{j=1}^k \beta_jc_j,$$
 see e.g. \cite{FarautKoranyi}, Lemma X.2.2.
 
 Let $c_1,\ldots, c_k$ now be a fixed Jordan frame of $J_\R$, 
 and define the following subspace of $J$:
 $$R_c=\left\{\sum_{j=1}^k \alpha_j c_j|\,\alpha_j\in \mathbb C\right\}$$
 Then,
 $$\mu^{-1}(0)=K\cdot R_c,$$
 since $K$ acts transitively on the set of Jordan frames of $J_\mathbb R$.
 
 Using that an orbit $G\cdot x$ is closed if and only if it meets
 the zero level set $\mu^{-1}(0)$ (see e.g. \cite{KempfNess}), this also shows
 that the orbit $G\cdot x$ is closed if and only if $x\in J$ is a 
 semi-simple element.
\end{rmk}

\section{Orbit closures}

For an $n$-tuple $x=(x_1,\ldots, x_n)\in J\times \ldots \times J$, 
let $A(x)$ denote again the subalgebra generated by $x_1,\ldots, x_n$.
Let $R(x)$ denote the radical of $A(x)$ and let $S(x)$ be a semi-simple
subalgebra of $A(x)$ satisfying $A(x)=S(x)+R(x)$.
Then there are unique elements $s_j\in S$, $r_j \in R$ such that 
$x_j=s_j+r_j$, and writing $s=(s_1,\ldots, s_n)$, $r=(r_1,\ldots, r_n)$
we get the decomposition $x=s+r$. This can be regarded as a
``Levi-decomposition'' of the $n$-tuple $x$.

\begin{rmk}
 Note that the decomposition of the $n$-tuple $x$ into $x=s+r$ is 
 in general not unique, but depends on the choice of the semi-simple
 subalgebra $S(x)$.
 
 However, in the case $n=1$ the subalgebra $A(x)$ is generated by one
 element and is thus associative. Consequently, decomposition
 $x=s+r$ is unique and is the usual decomposition of $x$ into its 
 semi-simple and nilpotent part. 
 
 More generally, the decomposition $x=s+r$ is unique whenever the elements
 of the radical $R(x)$ operator-commute with all elements in $A(x)$, i.e.
 $[A(x), R(x)]=0$.
\end{rmk}

\begin{prop}
 Let $x=s+r$ be a decomposition of an $n$-tuple $x\in J\times\ldots \times J$
 as above. 
 Then $\Aut(J)\cdot s$ is the unique closed orbit in the closure
 $\overline{\Aut(J)\cdot x}$ of the $\Aut(J)$-orbit through~$x$.
\end{prop}

\begin{proof}
 Let $\lambda:\mathbb K^\times \rightarrow \Aut(J)$ be a one-parameter
 subgroup such that $s\in J_0(\lambda)$ and $r\in J_{>0}(\lambda)$
 (c.f. Proposition~\ref{prop: ex of Jordan grading}).
 Then $s=\lim_{t\to 0}\lambda(t)\cdot x$ is contained in 
 $\overline{\Aut(J)\cdot x}$.
 
 The subalgebra $A(s)$ generated by $s_1,\ldots, s_n$ is contained
 in $S(x)$ since $s_j\in S(x)$ for each $j$. 
 Moreover, $x_j=s_j+r_j\in A(s)+R(x)$.
 As $S_1\oplus R(x)$ is a subalgebra of $J$ for any subalgebra~$S_1$ of 
 $S(x)$,
 we get $A(s)=S(x)$, and $A(s)$ is semi-simple.
 It follows that the orbit $\Aut(J)\cdot s$ is closed.
\end{proof}

\end{document}